\documentclass[10pt]{amsart}

\usepackage{amsmath,amsfonts, latexsym,graphicx, footmisc, amssymb, mathtools, enumerate}

\usepackage{hyperref}
\hypersetup{colorlinks=true, urlcolor=blue, citecolor=blue, linkcolor=blue}

\newcommand{\C}{{\mathbb C}}
\newcommand{\Z}{{\mathbb Z}}

\newcommand{\Q}{{\mathbb Q}}

\newcommand{\hyp}{{\mathbb H}}

\newcommand{\arrow}[1]{\stackrel{#1}{\longrightarrow}}
\newcommand{\Adot}{\mathbf A^\bullet}

\newcommand{\Fdot}{\mathbf F^\bullet}

\newtheorem{defn0}{Definition}[section]
\newtheorem{prop0}[defn0]{Proposition}
\newtheorem{conj0}[defn0]{Conjecture}
\newtheorem{thm0}[defn0]{Theorem}
\newtheorem{lem0}[defn0]{Lemma}
\newtheorem{corollary0}[defn0]{Corollary}
\newtheorem{example0}[defn0]{Example}
\newtheorem{remark0}[defn0]{Remark}
\newtheorem{question0}[defn0]{Question}
\newtheorem{exercise0}[defn0]{Exercise}

\newenvironment{defn}{\begin{defn0}}{\end{defn0}}
\newenvironment{prop}{\begin{prop0}}{\end{prop0}}

\newenvironment{lem}{\begin{lem0}}{\end{lem0}}

\newenvironment{rem}{\begin{remark0}\rm}{\end{remark0}}

\newcommand{\defref}[1]{Definition~\ref{#1}}

\newcommand{\lemref}[1]{Lemma~\ref{#1}}

\title{A note on the support and cosupport conditions for a perverse sheaf}


\author{David B. Massey}

\date{}

\begin{document}

\begin{abstract} We give a characterization of the support and cosupport conditions for a perverse sheaf in terms of the Whitney filtration. \end{abstract}

\maketitle




\section{Introduction}

Let $R$ be a regular, Noetherian ring with finite Krull dimension (e.g., $\Z$, $\Q$, or $\C$), and let $\Adot$ be a bounded, constructible complex of sheaves of $R$-modules on a complex analytic space $X$, i.e., let $\Adot\in D^b_c(X)$.

The complex of sheaves $\Adot$ is, by definition, a perverse sheaf (using the singular form ``sheaf'' is standard) if and only if it satisfies two conditions: the {\bf support} and {\bf cosupport} conditions. There are two equivalent well-known characterizations of these conditions, a characterization that does not refer to a Whitney stratification and one that does. See \cite{BBD}, \cite{inthom2}, \cite{kashsch}, \cite{dimcasheaves}, and \cite{schurbook}.

To give these two characterizations, let us first select a Whitney stratification $\mathcal S$ of $X$ with respect to which $\Adot$ is constructible and, for each $S\in\mathcal S$, let $r_S:S\hookrightarrow X$ denote the inclusion. Note that there is no requirement that the strata of $\mathcal S$ be connected, so in fact, by replacing the strata of $\mathcal S$ with a stratification where the strata are the unions of the strata of $\mathcal S$ of each given dimension, we may assume if we wish that there is (at most) one stratum of each dimension . 

We also need to define the $k$-th support and cosupport of $\Adot$. For all $x\in X$, let $j_x:\{x\}\hookrightarrow X$ denote the inclusion. The $k$-th support of $\Adot$ is
$$
\operatorname{supp}^k(\Adot):=\overline{\{x\in X\,|\, H^k(j_x^*\Adot)\neq 0\}}
$$
and the $k$-th cosupport is
$$
\operatorname{cosupp}^k(\Adot):=\overline{\{x\in X\,|\, H^k(j_x^!\Adot)\neq 0\}}.
$$

Now we give the standard descriptions of support and cosupport conditions on $\Adot$. The dimensions here are the complex dimensions and, as usual, a negative dimension indicates that a set is empty.
\medskip
\begin{defn}\label{defn:suppco} The support and cosupport conditions are defined in either/both of the following two ways:
\begin{enumerate}
\item 
\begin{itemize}
\item S1: (support) For all $k$, $\dim(\operatorname{supp}^{-k}(\Adot))\leq k$.
\smallskip
\item C1: (cosupport) For all $k$, $\dim (\operatorname{cosupp}^{k}(\Adot))\leq k$.
\end{itemize}
\bigskip
\item 
\begin{itemize}
\item S2: (support) For all $S\in\mathcal S$, for all $k>-\dim S$, $\mathbf H^k(r_S^*\Adot)=0$.
\smallskip
\item C2: (cosupport)  For all $S\in\mathcal S$, for all $k<-\dim S$, $\mathbf H^k(r_S^!\Adot)=0$.
\end{itemize}
\end{enumerate}
\end{defn}

\begin{rem} It is easy to see tht the two different characterizations of the support condition are equivalent, but it is significantly more difficult to to see that the two different cosupport conditions are equivalent. The two variants are connected by the following general result.

If $g : Y \hookrightarrow X$ is the inclusion of an orientable submanifold into another orientable manifold, and $r$ is the real codimension of $Y$ in $X$, and $\Fdot \in \bold D^b_c(X)$ has locally constant
cohomology on $X$, then $g^!\Fdot$ has locally constant cohomology on $Y$
and $g^*\Fdot[-r] \cong g^!\Fdot$.

\smallskip
Referring now to \defref{defn:suppco}, if we let $\Fdot=r_S^!\Adot$, let $x\in S$, let $g:\{x\}\hookrightarrow S$ be the inclusion, and $r=2\dim S$ (where $\dim$ means the complex dimension), we obtain
$$
\mathbf H^k(r_S^!\Adot)_x\cong H^k(g^*r_S^!\Adot)\cong H^k(g^!r_S^!\Adot[2\dim S])\cong H^{k+2\dim S}(j_x^!\Adot).
$$
Therefore, $\mathbf H^k(r_S^!\Adot)=0$ if and only if, for all $x\in S$, $H^{k+2\dim S}(j_x^!\Adot)=0$  and so:

\smallskip

\noindent for all $k<-\dim S$, $\mathbf H^k(r_S^!\Adot)=0$ if and only if, for all $k<\dim S$, for all $x\in S$,  $H^{k}(j_x^!\Adot)=0$.

That is, for all $S\in\mathcal S$, for all $k<-\dim S$, $\mathbf H^k(r_S^!\Adot)=0$ if and only if, for all $k$, $\operatorname{cosupp}^{k}(\Adot)\subseteq \bigcup_{\dim S\leq k}S$.
\end{rem}

\vskip 0.2in

There is only one result in this short paper, one that may be of use as a lemma for other results; we give the support and cosupport conditions in terms of the open and closed filtrations of $X$ which come from the stratification $\mathcal S$.

\medskip

\section{support and cosupport in terms of filtrations}

We continue with all of our notation from the introduction. In particular, we have a Whitney stratification $\mathcal S$ of $X$ with respect to which $\Adot$ is constructible and, for each $S\in\mathcal S$,  $r_S:S\hookrightarrow X$ is the inclusion.

\smallskip

We define the {\bf upper filtration} of $X$ by: for all $m\geq0$, let $U^m:=\bigcup_{S\in\mathcal S, \dim S\geq m}S$. We define the {\bf lower filtration} of $X$ by: for all $m\geq 0$, let $L^m:=\bigcup_{S\in\mathcal S, \dim S\leq m}S$.

\smallskip

Then, our new characterizations of the support and cosupport condition are:
\smallskip
\begin{lem}\label{lem:main} For all $m\geq 0$, let $u_m: U^m\hookrightarrow X$ and $\ell_m: L^m\hookrightarrow X$ denote the inclusions.
\begin{itemize}
\item The support condition is equivalent to: (new support) for all $m\geq 0$, for all $k>-m$, $\mathbf H^k(u_m^*\Adot)=0$.
\smallskip
\item The cosupport condition is equivalent to: (new cosupport) for all $m\geq0$, for all $k<-m$, $\mathbf H^k(\ell_m^!\Adot)=0$.
\end{itemize}
\end{lem}
\begin{proof} 

\smallskip

\phantom{filler}

\smallskip

\noindent Support:

\smallskip

This is easy. Both (S2) and the new support condition are trivially equivalent to: for all $k$, for all $S\in\mathcal S$ such that $\dim S>-k$, for all $x\in S$, $H^k(j^*_x\Adot)=0$.

\bigskip

\noindent Cosupport:

\smallskip

For notational convenience, we assume that our Whitney stratification has, at most, one stratum for each dimension. We can do this by replacing the original individual strata of dimension $m$ by one stratum which is the union of all of the original $m$-dimensional strata. This has no effect on Version 2 in \defref{defn:suppco} as each original $m$-dimensional stratum is an open subset of the union of all of the $m$-dimensional strata (since the union is locally connected). We denote the unique $m$-dimensional stratum by $S^m$.

Now, let $q_m:S^m\hookrightarrow L^m$ denote the inclusion; as $S^m$ is open in $L^m$, we have a natural isomorphism $q_m^*\cong q_m^!$. We write $r_m$  in place of $r_{S^m}$ for the inclusion of $S^m$ into $X$. Thus, $r_m=\ell_m q_m$.

\smallskip

Note that with the above assumption and notation, the cosupport condition (C2) becomes:  \newline 
\noindent (C2) \ for all $m$, for all $k<-m$, $\mathbf H^k(r_m^!\Adot)=0$.

\bigskip

\noindent New cosupport $\implies$ (C2):

\smallskip

Suppose the new cosupport condition holds, so that, for all $m\geq 0$, for all $k<-m$, $\mathbf H^k(\ell_m^!\Adot)=0$. Then, it is immediate that,  for all $m\geq 0$, for all $k<-m$, 
$$0=\mathbf H^k(q_m^*\ell_m^!\Adot)\cong \mathbf H^k(q_m^!\ell_m^!\Adot)\cong \mathbf H^k(r_m^!\Adot),
$$
which is (C2).

\bigskip

\noindent (C2) $\implies$ new cosupport:

\smallskip

Suppose that (C2) holds. So assume that, for all $m\geq 0$, for all $k<-m$, $\mathbf H^k(r_m^!\Adot)=0$. We shall prove by induction on $m$ that, for all $m\geq0$, for all $k<-m$, $\mathbf H^k(\ell_m^!\Adot)=0$.

\bigskip

\noindent $m=0$: 

\smallskip

Since $S^0=L^0$, (C2) immediately implies that, for all $k<0$, $\mathbf H^k(\ell_{0}^!\Adot)=0$.

\bigskip

\noindent Inductive step: 

\smallskip

Now assume that  $m_0\geq 1$ and that the new cosupport conditions holds for all $m$ such that $0\leq m\leq m_0-1$. We wish to show that, for all $k<-m_0$, $\mathbf H^k(\ell_{m_0}^!\Adot)=0$.

Let $p_{m_0}: L^{m_0-1}=L^{m_0}\backslash S^{m_0}\hookrightarrow L^{m_0}$ denote the closed inclusion. Then we have the fundamental distinguished triangle
$$
(p_{m_0})_!p_{m_0}^!(\ell_{m_0}^!\Adot)\rightarrow \ell_{m_0}^!\Adot\rightarrow (q_{m_0})_*q_{m_0}^*(\ell_{m_0}^!\Adot)\arrow{[1]} 
$$
and the associated long exact sequence on stalk cohomology.

Suppose that $k<-m_0$ and $x\in L^{m_0}$. We will show that the stalk cohomology $H^k( \ell_{m_0}^!\Adot)_x=0$ by showing that $H^k\big((p_{m_0})_!p_{m_0}^!(\ell_{m_0}^!\Adot)\big)_x=0$ and $H^k\big((q_{m_0})_*q_{m_0}^*(\ell_{m_0}^!\Adot)\big)_x=0$.

\smallskip

As $\ell_{m_0}p_{m_0}=\ell_{m_0-1}$ and $k<-m_0<-(m_0-1)$, we have
$$
H^k\big((p_{m_0})_!p_{m_0}^!(\ell_{m_0}^!\Adot)\big)_x\cong H^k\big((p_{m_0})_!\ell_{m_0-1}^!\Adot\big)_x,
$$
which equals $0$ by our inductive step if $x\in L^{m_0-1}$ and equals $0$ otherwise since $(p_{m_0})_!$ is the extension by zero.

\smallskip

As $q_{m_0}^*\cong q_{m_0}^!$ and $r_{m_0}=\ell_{m_0} q_{m_0}$, we have  $H^k\big((q_{m_0})_*q_{m_0}^*(\ell_{m_0}^!\Adot)\big)_x\cong H^k\big((q_{m_0})_*r_{m_0}^!\Adot\big)_x$, and note that, by (C2),  for all $k<-m_0$, $\mathbf H^k(r_{m_0}^!\Adot)=0$. By constructibility, there is an open neighborhood $U$ of $x$ in $L^{m_0}$ such that
$$
H^k\big((q_{m_0})_*r_{m_0}^!\Adot\big)_x\cong \hyp^k\big(U; \big(r_{m_0}^!\Adot\big)_{|_U}\big),
$$
but now either the canonical injective resolution of $\big(r_{m_0}^!\Adot\big)_{|_U}$ or the $E_2$ spectral sequence for hypercohomology tells us that this is $0$.
\end{proof}

\bigskip

Why did we want to prove that there is a third way of describing the support and cosupport conditions, as is given in \lemref{lem:main}? 

\medskip

It is our hope that these new support and cosupport conditions will become as equally well-known as the two versions given in \defref{defn:suppco}, mainly for the new version of the cosupport condition. We, ourselves, and others have instead had to essentially prove that (C2) implies the new cosupport condition in the midst of other results; this is the case, for instance, in the proof of Theorem 3.1 in \cite{MPT} where $\Adot$ is the perverse sheaf of shifted vanishing cycles along a function $f$. 

Using the notation from \lemref{lem:main}, we can now easily prove the following proposition.

\begin{prop}\label{prop:main} Let $\mathcal S$ be a Whitney stratification of $X$ with respect to which $\Adot$ is constructible. For all $m\geq 0$, let $U^m:=\bigcup_{S\in\mathcal S, \dim S\geq m}S$. Suppose that $\Adot$ satisfies the cosupport condition.

Then, for all $m\geq 0$,
\begin{enumerate}
\item for all $k\leq -m-2$, the canonical morphism yields an isomorphism $\hyp^k(X;\,\Adot)\cong \hyp^k(U^{m+1};\, \Adot_{|_{U^{m+1}}})$, and
\smallskip
\item the canonical morphism yields an injection $\hyp^{-m-1}(X;\,\Adot)\hookrightarrow \hyp^{-m-1}(U^{m+1};\, \Adot_{|_{U^{m+1}}})$.
\end{enumerate}
\end{prop}
\begin{proof} As before, for all $m\geq 0$, let $L^m:=\bigcup_{S\in\mathcal S, \dim S\leq m}S$, and let $u_m: U^m\hookrightarrow X$ and $\ell_m: L^m\hookrightarrow X$ denote the inclusions. 

Consider the distinguished triangle 
$$
(\ell_m)!\ell_m^!\Adot\rightarrow\Adot\rightarrow (u_{m+1})*u_{m+1}^*\Adot\arrow{[1]}.
$$

Then, combining \lemref{lem:main} with the long exact sequence on hypercohomology yields the stated conclusions (once again using, as in the proof of the lemma, that zero sheaf cohomology below a given dimension implies zero hypercohomology below that dimension).
\end{proof}

\bibliographystyle{plain}

\bibliography{Masseybib}

\end{document}